\documentclass[12pt]{article}
\usepackage[]{amsmath,amssymb}
\usepackage{amscd}
\usepackage{latexsym}
\usepackage{cite}
\usepackage{amsthm}

\newtheorem{definition}{Definition}[section]
\newtheorem{theorem}[definition]{Theorem}
\newtheorem{lemma}[definition]{Lemma}
\newtheorem{corollary}[definition]{Corollary}
\newtheorem{remark}[definition]{Remark}
\newtheorem{example}[definition]{Example}

\newtheorem{proposition}[definition]{Proposition}

\begin{document}
\title{\bf 
The algebra $U^+_q$ and its alternating \\central extension $\mathcal U^+_q$}
 \author{
Paul Terwilliger 
}
\date{}

\maketitle
\begin{abstract} Let $U^+_q$ denote the positive part of the quantized  enveloping algebra 
$U_q(\widehat{\mathfrak{sl}}_2)$.
The algebra $U^+_q$ has a presentation involving two generators $W_0$, $W_1$ and two relations, called the
$q$-Serre relations. In 1993 I. Damiani obtained a PBW basis for $U^+_q$, consisting of some elements
$\lbrace E_{n \delta+ \alpha_0} \rbrace_{n=0}^\infty$,
$\lbrace E_{n \delta+ \alpha_1} \rbrace_{n=0}^\infty$,
$\lbrace E_{n \delta} \rbrace_{n=1}^\infty$.
In 2019 we introduced the alternating central extension $\mathcal U^+_q$ of $U^+_q$. We defined $\mathcal U^+_q$ by generators and relations.
The generators, said to be alternating, are denoted
$\lbrace \mathcal W_{-k}\rbrace_{k=0}^\infty$, $\lbrace \mathcal  W_{k+1}\rbrace_{k=0}^\infty$, 
$ \lbrace \mathcal G_{k+1}\rbrace_{k=0}^\infty$, 
$\lbrace \mathcal {\tilde G}_{k+1}\rbrace_{k=0}^\infty$. Let $\langle \mathcal W_0, \mathcal W_1 \rangle$ denote
the subalgebra of $\mathcal U^+_q$ generated by $\mathcal W_0$, $\mathcal W_1$.
It is known that there exists an algebra isomorphism
$U^+_q\to \langle \mathcal W_0, \mathcal W_1 \rangle$ 
 that sends $W_0 \mapsto \mathcal W_0$ and $W_1 \mapsto \mathcal W_1$.
Via this isomorphism we identify $U^+_q$ with $\langle \mathcal W_0, \mathcal W_1 \rangle$.
In our main result, we express the Damiani PBW basis elements in terms of the alternating generators.
We give the answer in terms of generating functions.
\bigskip


 \end{abstract}
 
 \section{Introduction}
 \noindent We will be discussing the positive part $U^+_q$ of the quantized enveloping algebra $U_q(\widehat{\mathfrak{sl}}_2)$.
 The algebra $U^+_q$ is associative an infinite-dimensional. 
It has a presentation involving
two generators $W_0$, $W_1$ and two relations, called the $q$-Serre relations:
 \begin{align*}
\lbrack W_0, \lbrack W_0, \lbrack W_0, W_1\rbrack_q \rbrack_{q^{-1}} \rbrack =0, \qquad \qquad 
\lbrack W_1, \lbrack W_1, \lbrack W_1, W_0\rbrack_q \rbrack_{q^{-1}}\rbrack = 0.
\end{align*}
In \cite{damiani} I. Damiani 
  obtained a Poincar\'e-Birkhoff-Witt (or PBW)
basis for $U^+_q$. The PBW basis elements are denoted
\begin{align}
\lbrace E_{n \delta+ \alpha_0} \rbrace_{n=0}^\infty,
\qquad \quad 
\lbrace E_{n \delta+ \alpha_1} \rbrace_{n=0}^\infty,
\qquad \quad 
\lbrace E_{n \delta} \rbrace_{n=1}^\infty.
\label{eq:UpbwIntro}
\end{align}
 We will be discussing the generating functions
 \begin{align*}
&E^-(t) = \sum_{n=0}^\infty  E_{n\delta+\alpha_0} t^n, \qquad \quad
E^+(t) = \sum_{n=0}^\infty  E_{n\delta+\alpha_1} t^n,
\\
&E(t) = \sum_{n=0}^\infty E_{n\delta} t^n,       \qquad \qquad E_{0\delta} = -(q-q^{-1})^{-1}.
\end{align*}
  In \cite{alternating}
we introduced a type of element in $U^+_q$, said to be alternating. 
By
\cite[Lemma 5.11]{alternating},
each alternating element commutes with exactly one of
$W_0$, $W_1$, $\lbrack W_1, W_0\rbrack_q$, $\lbrack W_0, W_1\rbrack_q$. This gives four types of
alternating elements,
denoted
\begin{align*}
\lbrace  W_{-k}\rbrace_{k\in \mathbb N}, \quad
\lbrace  W_{k+1}\rbrace_{k\in \mathbb N}, \quad
\lbrace  G_{k+1}\rbrace_{k\in \mathbb N}, \quad
\lbrace {\tilde G}_{k+1}\rbrace_{k\in \mathbb N}.
\end{align*}
By 
\cite[Lemma 5.11]{alternating} the alternating elements of
each type mutually commute.
\medskip

\noindent 
We obtained the alternating elements in the following way.
Consider the free algebra $\mathbb V$ on two generators $x,y$.
The standard (linear) basis for
$\mathbb V$ consists of the words in $x,y$.
In
\cite{rosso1, rosso} M. Rosso introduced
an algebra structure on $\mathbb V$, called a
$q$-shuffle algebra.
For $u,v\in \lbrace x,y\rbrace$ their
$q$-shuffle product is
$u\star v = uv+q^{\langle u,v\rangle }vu$, where
$\langle u,v\rangle =2$
(resp. $\langle u,v\rangle =-2$)
if $u=v$ (resp.
 $u\not=v$).
Rosso gave an injective algebra homomorphism $\natural$
from $U^+_q$ into the $q$-shuffle algebra
${\mathbb V}$, that sends $W_0\mapsto x$ and $W_1\mapsto y$.
By
\cite[Definition~5.2]{alternating}
the map $\natural$ sends
\begin{align*}
&W_0 \mapsto x, \qquad W_{-1} \mapsto xyx, \qquad W_{-2} \mapsto xyxyx, \qquad \ldots
\\
&W_1 \mapsto y, \qquad W_{2} \mapsto yxy, \qquad W_{3} \mapsto yxyxy, \qquad \ldots
\\
&G_{1} \mapsto yx, \qquad G_{2} \mapsto yxyx,  \qquad G_3 \mapsto yxyxyx, \qquad \ldots
\\
&\tilde G_{1} \mapsto xy, \qquad \tilde G_{2} \mapsto 
xyxy,\qquad \tilde G_3 \mapsto xyxyxy, \qquad \ldots
\end{align*}
In \cite{alternating} we used $\natural$
to obtain many relations involving the 
 alternating elements; the main relations are listed in Definition \ref{def:Aq} below 
 and  \cite[Proposition~8.1]{alternating}. In \cite[Section~11]{alternating} we described how the
 alternating elements are related to the elements \eqref{eq:UpbwIntro}.
 \medskip

\noindent 
 In \cite{altCE}
we defined
an algebra  $\mathcal U^+_q$  by 
generators and relations in the following way. The generators, said to be alternating, are denoted
\begin{align*}
\lbrace \mathcal W_{-k}\rbrace_{k \in \mathbb N}, \quad 
\lbrace  \mathcal W_{k+1}\rbrace_{k \in \mathbb N}, \quad
\lbrace  \mathcal G_{k+1}\rbrace_{k\in \mathbb N}, \quad
\lbrace \mathcal {\tilde G}_{k+1}\rbrace_{k \in \mathbb N}.
\end{align*}
The relations are the ones in Definition \ref{def:Aq}.
By construction there 
 exists a surjective algebra homomorphism
$ \mathcal U^+_q \to U^+_q$ that sends
\begin{align*}
\mathcal W_{-k} \mapsto W_{-k},\quad \qquad
\mathcal W_{k+1} \mapsto W_{k+1},\quad \qquad
\mathcal G_{k} \mapsto G_{k},\quad \qquad
\mathcal {\tilde G}_{k} \mapsto \tilde G_{k}
\end{align*}
for $k \in \mathbb N$.
In a moment, we will see that this map is not injective.
Denote the ground field by $\mathbb F$ and
let $\lbrace z_n\rbrace_{n=1}^\infty$
denote mutually commuting indeterminates.
Let
$\mathbb F\lbrack z_1, z_2,\ldots\rbrack$
denote the algebra consisting of the polynomials in
$z_1, z_2,\ldots $ that have all coefficients in 
$\mathbb F$. 
For notational convenience define $z_0=1$.
In \cite[Lemma~3.6, Theorem~5.17]{altCE} we displayed an algebra isomorphism
$\varphi:
\mathcal U^+_q \to U^+_q \otimes 
\mathbb F \lbrack z_1, z_2,\ldots\rbrack$ that sends
\begin{align*}
\mathcal W_{-n} &\mapsto \sum_{k=0}^n W_{k-n} \otimes z_k,
\quad \qquad \qquad 
\mathcal W_{n+1} \mapsto \sum_{k=0}^n W_{n+1-k} \otimes z_k,
\\
\mathcal G_{n} &\mapsto \sum_{k=0}^n G_{n-k} \otimes z_k,
\quad \qquad \qquad
\mathcal {\tilde G}_{n} \mapsto \sum_{k=0}^n \tilde G_{n-k} \otimes z_k
\end{align*}
for $n \in \mathbb N$. In particular, $\varphi$ sends
$\mathcal W_0 \mapsto W_0 \otimes 1$ and
$\mathcal W_1 \mapsto W_1 \otimes 1$. Following \cite{altCE} we call $\mathcal U^+_q$
the alternating central extension of $U^+_q$.
\medskip

\noindent
In \cite{altCE} we obtained the following results about the center $\mathcal Z$  of $\mathcal U^+_q$. By \cite[Lemma~5.10]{altCE} the map $\varphi$ sends $\mathcal Z \mapsto 1 \otimes \mathbb F\lbrack z_1, z_2,\ldots\rbrack$.
For $n\geq 1$ define
\begin{align*}
 \mathcal Z^\vee_n= \sum_{k=0}^n \mathcal G_k \mathcal {\tilde G}_{n-k} q^{n-2k}
-  q \sum_{k=0}^{n-1} \mathcal W_{-k} \mathcal W_{n-k} q^{n-1-2k}.
\end{align*}
For notational convenience define $\mathcal Z^\vee_0=1$.
By \cite[Definition~5.5, Proposition~6.2]{altCE} the subalgebra $\mathcal Z$
is generated by $\lbrace \mathcal Z^\vee_n\rbrace_{n=1}^\infty$.
By \cite[Lemma~5.4]{altCE}, for $n\in \mathbb N$ the map $\varphi$ sends $\mathcal Z^\vee_n \mapsto 1 \otimes z^\vee_n$
where $z^\vee_n = \sum_{k=0}^n z_k z_{n-k} q^{n-2k}$.
By \cite[Corollary~6.3]{altCE} the elements
$\lbrace \mathcal Z^\vee_n\rbrace_{n=1}^\infty$ are algebraically
independent.
\medskip

\noindent 
Let $\langle \mathcal W_0, \mathcal W_1\rangle$
denote the subalgebra of 
 $\mathcal U^+_q$  generated by
$\mathcal W_0, \mathcal W_1$.
By \cite[Proposition~6.4]{altCE} there exists an algebra isomorphism $U^+_q \to  \langle \mathcal W_0, \mathcal W_1\rangle$ that sends
$W_0 \mapsto \mathcal W_0$ and $W_1 \mapsto \mathcal W_1$.
 By \cite[Proposition~6.5]{altCE}  the multiplication map
\begin{align*}
\langle \mathcal W_0,\mathcal W_1\rangle \otimes \mathcal Z &\to
	       \mathcal U^+_q 
	       \\
w \otimes z &\mapsto      wz            
\end{align*}
is an algebra isomorphism.
By \cite[Theorem~10.2]{altCE} the alternating generators in order
\begin{align}
\lbrace \mathcal W_{-k} \rbrace_{k \in \mathbb N}, \qquad 
\lbrace \mathcal G_{k+1} \rbrace_{k\in \mathbb N}, \qquad  
\lbrace \mathcal {\tilde G}_{k+1} \rbrace_{k\in \mathbb N}, \qquad  
\lbrace \mathcal W_{k+1} \rbrace_{k\in \mathbb N}
\label{eq:PBW4}
\end{align}
give a PBW basis for $\mathcal U^+_q$. 
\medskip

\noindent  We now summarize the main results of the present paper. 
For the rest of this section, we identify the algebra $U^+_q$ with $\langle \mathcal W_0, \mathcal W_1 \rangle$ via the isomorphism mentioned above.
Our goal is to elegantly express the elements  \eqref{eq:UpbwIntro} in terms of the alternating generators for $\mathcal U^+_q$.
To accomplish the goal, we first adjust the PBW basis  \eqref{eq:PBW4} by modifying the ordering as follows.
We show that the alternating generators in order
\begin{align}
\label{eq:GWWG}
\lbrace \mathcal G_{k+1} \rbrace_{k\in \mathbb N}, \qquad  
\lbrace \mathcal W_{-k} \rbrace_{k \in \mathbb N}, \qquad 
\lbrace \mathcal W_{k+1} \rbrace_{k\in \mathbb N}, \qquad 
\lbrace \mathcal {\tilde G}_{k+1} \rbrace_{k\in \mathbb N}
\end{align}
give a PBW basis for $\mathcal U^+_q$. This PBW basis induces a basis for $\mathcal U^+_q$, in which we will express the elements \eqref{eq:UpbwIntro}.
We give our answer in terms of generating functions. Define
  \begin{align*}
&\mathcal W^-(t) = \sum_{n=0}^\infty \mathcal W_{-n} t^n,
\qquad \qquad  
\mathcal W^+(t) = \sum_{n=0}^\infty \mathcal W_{n+1} t^n,
\\
&\mathcal G(t) = \sum_{n=0}^\infty \mathcal G_n t^n,
\qquad \quad 
\mathcal {\tilde G}(t) = \sum_{n=0}^\infty \mathcal {\tilde G}_n t^n, \qquad \quad \mathcal G_0 = \mathcal {\tilde G}_0 = 1.
\end{align*}
Further define
$\mathcal Z^\vee(t) = \sum_{n\in \mathbb N} \mathcal Z^\vee_n t^n$. By construction
\begin{align*}
\mathcal Z^\vee(t) &= 
 \mathcal G(q^{-1}t) \mathcal{\tilde G}(qt) -qt
\mathcal W^{-} (q^{-1}t)\mathcal W^+(qt).
\end{align*}
We obtain the factorization
\begin{align*}
\mathcal Z^\vee(t)  =-(q-q^{-1}) \mathcal {\tilde G}(q^{-1}t) E(\xi t) \mathcal {\tilde G}(qt),
\end{align*}
\noindent where 
$\xi = -q^2 (q-q^{-1})^{-2}$. Using this factorization we obtain
\begin{align*}
E^-(t) &= \mathcal W^-(q^{-1} \xi^{-1} t) \bigl( \mathcal {\tilde G}(q^{-1} \xi^{-1} t)\bigr)^{-1}, 
 \\
E^+(t) &=    \mathcal W^+(q \xi^{-1} t) \bigl( \mathcal {\tilde G}(q \xi^{-1} t)\bigr)^{-1}, 
\\
E(t) &= -\,\frac{\mathcal Z^\vee(\xi^{-1}t) \bigl( \mathcal {\tilde G}(q^{-1} \xi^{-1} t)\bigr)^{-1} \bigl( \mathcal {\tilde G}(q \xi^{-1} t)\bigr)^{-1}}{q-q^{-1}}.
 \end{align*}
 The above three equations effectively express the elements  \eqref{eq:UpbwIntro} in the basis for $\mathcal U^+_q$ induced by the PBW basis \eqref{eq:GWWG}.
 Using the above three equations and the relations in Definition \ref{def:Aq}, we recover the previously known relations between $E^\pm (t)$, $E(t)$.
\medskip

\noindent The paper is organized as follows. Section 2 contains some preliminaries. In Section 3 we recall the definition and basic facts about $U^+_q$.
In Section 4 we recall the PBW basis for $U^+_q$ due to Damiani, and give the corresponding reduction rules.
In Section 5 we express these reduction rules in terms of the generating functions $E^\pm(t)$, $E(t)$.
In Section 6 we recall the definition and basic facts about $\mathcal U^+_q$. In Section 7 we express
the defining relations for $\mathcal U^+_q$ in terms of the generating functions $\mathcal W^\pm (t)$, $\mathcal G(t)$, $\mathcal {\tilde G}(t)$.
 In Section 8 we obtain a PBW basis for $\mathcal U^+_q$,
and give the corresponding reduction rules. In Section 9 we describe the center of $\mathcal U^+_q$ and recall the generating function $\mathcal Z^\vee(t)$.
In Section 10 we compare the generating functions  $E^\pm(t)$, $E(t)$ with the generating functions $\mathcal W^\pm (t)$, $\mathcal G(t)$, $\mathcal {\tilde G}(t)$.
In Section 11 we obtain a factorization of  $\mathcal Z^\vee(t)$.
In Section 12 we express  $E^\pm(t)$, $E(t)$ in terms of $\mathcal W^\pm (t)$, $\mathcal G(t)$, $\mathcal {\tilde G}(t)$.
In Appendix A we recall an earlier PBW basis for $\mathcal U^+_q$ and give the corresponding reduction rules.

 \section{Preliminaries}
 \noindent  We now begin our formal argument. Throughout the paper, the following notational conventions are in effect.
Recall the natural numbers $\mathbb N= \lbrace 0,1,2,\ldots \rbrace$ and integers $\mathbb Z=\lbrace 0,\pm 1, \pm 2,\ldots \rbrace$. Let $\mathbb F$ denote a field.
Every vector space and tensor product discussed in this paper is over $\mathbb F$.
Every algebra discussed in this paper is associative, over $\mathbb F$, and has a multiplicative identity. A subalgebra has the same multiplicative identity as the parent algebra.
 Let $\mathcal A$ denote an algebra. By an {\it automorphism} of $\mathcal A$
we mean an algebra isomorphism $\mathcal A\rightarrow \mathcal A$. The algebra $\mathcal A^{\rm opp}$ consists of the vector space $\mathcal A$ and the multiplication map $\mathcal A \times \mathcal A \rightarrow \mathcal A$, $(a,b)\to ba$.
By an {\it antiautomorphism} of $\mathcal A$ we mean an algebra isomorphism $\mathcal A \rightarrow \mathcal A^{\rm opp}$.
\medskip

\noindent We will be discussing generating functions. Let $\mathcal A$ denote an algebra and let $t$ denote an indeterminate.
For a sequence $\lbrace a_n \rbrace_{n \in \mathbb N}$ of elements in $\mathcal A$, the corresponding generating function is
\begin{align*}
 a(t) = \sum_{n \in \mathbb N} a_n t^n.
 \end{align*} 
 The above sum is formal; issues of convergence are not considered.
 We call $a(t)$ the {\it generating function over $\mathcal A$ with coefficients $\lbrace a_n \rbrace_{n \in \mathbb N}$}.
  For generating functions
 $a(t)=\sum_{n \in \mathbb N} a_n t^n$ and
 $b(t) = \sum_{n \in \mathbb N} b_n t^n$ over $\mathcal A$, their product $a(t)b(t)$ is the generating function $\sum_{n \in \mathbb N}c_n t^n$  such that 
 $c_n = \sum_{i=0}^n a_i b_{n-i}$ for $n\in \mathbb N$.
 The set of generating functions over $\mathcal A$ forms an algebra.
 The following result is readily checked.
 \begin{lemma} \label{lem:inv} Let $\mathcal A $ denote an algebra. A generating function $a(t)=\sum_{n \in \mathbb N} a_n t^n$ over $\mathcal A$ is invertible if and only if $a_0$ is invertible in $\mathcal A$. In this case
 $(a(t))^{-1}=\sum_{n \in \mathbb N}b_n t^n$ where
 $ b_0 = a^{-1}_0$ and for $n\geq 1$,
 \begin{align*}
 b_n = - a_0^{-1}\sum_{k=1}^n a_k b_{n-k}.
 \end{align*}
 \end{lemma}
 
 \begin{example}
\label{ex:GvsV}
\rm Referring to Lemma \ref{lem:inv}, assume that $a_0=1$. Then
\begin{align*}
&b_0=1, \qquad \qquad \qquad
b_1 = -a_1, \qquad \qquad \qquad
b_2 = a^2_1 -a_2,
\\
&b_3 = 2 a_1 a_2 -a^3_1 - a_3, \qquad \qquad
b_4 = a^4_1 +2a_1a_3+a^2_2-3a^2_1a_2-a_4.
\end{align*}
\end{example}

 \begin{definition}\label{def:pbw}
 \rm 
(See \cite[p.~299]{damiani}.)
Let $ \mathcal A$ denote an algebra. A {\it Poincar\'e-Birkhoff-Witt} (or {\it PBW}) basis for $\mathcal A$
consists of a subset $\Omega \subseteq \mathcal A$ and a linear order $<$ on $\Omega$
such that the following is a basis for the vector space $\mathcal A$:
\begin{align}
\label{eq:basis}
a_1 a_2 \cdots a_n \qquad n \in \mathbb N, \qquad a_1, a_2, \ldots, a_n \in \Omega, \qquad
a_1 \leq a_2 \leq \cdots \leq a_n.
\end{align}
We interpret the empty product as the multiplicative identity in $\mathcal A$.
\end{definition}

\begin{definition}\rm \label{lem:rr} We refer to the PBW basis $\Omega$, $<$ from Definition \ref{def:pbw}. For any ordered pair $a,b$ of elements in $\Omega$ such that $a>b$, the corresponding
{\it reduction rule} is the equation that expresses the product $ab$ as a linear combination of the basis elements from \eqref{eq:basis}. The reduction rule is called {\it trivial} whenever $a$, $b$ commute.
\end{definition}

\begin{definition}\label{def:poly} \rm
Let $\lbrace z_n \rbrace_{n=1}^\infty$ denote mutually commuting indeterminates. Let $\mathbb F \lbrack z_1, z_2, \ldots \rbrack$ denote
the algebra consisting of the polynomials in $z_1, z_2, \ldots $ that have all coefficients in $\mathbb F$.
For notational convenience, define $z_0=1$.
\end{definition}

  \noindent Throughout the paper, we fix a nonzero $q \in \mathbb F$
that is not a root of unity.
Recall the notation
\begin{align*}
\lbrack n\rbrack_q = \frac{q^n-q^{-n}}{q-q^{-1}}
\qquad \qquad n \in \mathbb N.
\end{align*}

\section{The algebra $U^+_q$}
In this section we recall the algebra $ U^+_q$.
\medskip

\noindent For elements $X, Y$ in any algebra, define their
commutator and $q$-commutator by 
\begin{align*}
\lbrack X, Y \rbrack = XY-YX, \qquad \qquad
\lbrack X, Y \rbrack_q = q XY- q^{-1}YX.
\end{align*}
\noindent Note that 
\begin{align*}
\lbrack X, \lbrack X, \lbrack X, Y\rbrack_q \rbrack_{q^{-1}} \rbrack
= 
X^3Y-\lbrack 3\rbrack_q X^2YX+ 
\lbrack 3\rbrack_q XYX^2 -YX^3.
\end{align*}

\begin{definition} \label{def:U} \rm
(See \cite[Corollary~3.2.6]{lusztig}.) 
Define the algebra $U^+_q$ by generators $W_0$, $W_1$ and relations
\begin{align}
\label{eq:qOns1}
&\lbrack W_0, \lbrack W_0, \lbrack W_0, W_1\rbrack_q \rbrack_{q^{-1}} \rbrack =0,
\\
\label{eq:qOns2}
&\lbrack W_1, \lbrack W_1, \lbrack W_1, W_0\rbrack_q \rbrack_{q^{-1}}\rbrack = 0.
\end{align}
We call $U^+_q$ the {\it positive part of $U_q(\widehat{\mathfrak{sl}}_2)$}.
The relations \eqref{eq:qOns1}, \eqref{eq:qOns2}  are called the {\it $q$-Serre relations}.
\end{definition}

\noindent We mention some symmetries of $U^+_q$. 

\begin{lemma}
\label{lem:aut} There exists an automorphism $\sigma$ of $U^+_q$ that sends $W_0 \leftrightarrow W_1$.
Moreover $\sigma^2 = {\rm id}$, where ${\rm id}$ denotes the identity map.
\end{lemma}

\begin{lemma}\label{lem:antiaut} {\rm (See \cite[Lemma~2.2]{catalan}.)}
There exists an antiautomorphism $\dagger$ of $U^+_q$ that fixes each of $W_0$, $W_1$.
 Moreover $\dagger^2={\rm id}$.
\end{lemma}

\begin{lemma} The maps $\sigma$, $\dagger$ commute.
\end{lemma}
\begin{proof} This is readily checked.
\end{proof}
 
\begin{definition}\label{def:tauA} \rm Let $\tau$ denote the composition of $\sigma$ and $\dagger$. Note that $\tau$ is an antiautomorphism of $U^+_q$ that sends
$W_0 \leftrightarrow W_1$. We have $\tau^2 = {\rm id}$.
\end{definition}

\section{A PBW basis for $U^+_q$}

\noindent In \cite{damiani}, Damiani obtained a PBW basis for $U^+_q$ that involves some elements
\begin{align}
\lbrace E_{n \delta+ \alpha_0} \rbrace_{n=0}^\infty,
\qquad \quad 
\lbrace E_{n \delta+ \alpha_1} \rbrace_{n=0}^\infty,
\qquad \quad 
\lbrace E_{n \delta} \rbrace_{n=1}^\infty.
\label{eq:Upbw}
\end{align}
These elements are recursively defined  as follows.  
\begin{align}
E_{\alpha_0} = W_0, \qquad \qquad
E_{\alpha_1} = W_1, \qquad \qquad
E_{\delta} = q^{-2}W_1W_0-W_0W_1,
\label{eq:BAalt}
\end{align}
and for $n\geq 1$,
\begin{align}
&
E_{n \delta+\alpha_0} =
\frac{
\lbrack E_\delta, E_{(n-1)\delta+ \alpha_0} \rbrack
}
{q+q^{-1}},
\qquad \qquad
E_{n \delta+\alpha_1} =
 \frac{
 \lbrack
 E_{(n-1)\delta+ \alpha_1},
 E_\delta
 \rbrack
 }
 {q+q^{-1}},
 \label{eq:dam1introalt}
 \\
 &
 \qquad \qquad
 E_{n \delta} =
 q^{-2}  E_{(n-1)\delta+\alpha_1} W_0
 - W_0 E_{(n-1)\delta+\alpha_1}.
 \label{eq:dam2introalt}
\end{align}


\begin{proposition}
    \label{prop:PBWbasis}
    {\rm (See \cite[p.~308]{damiani}.)}
A PBW basis for $U^+_q$ is obtained by the elements
{\rm \eqref{eq:Upbw}}
  in the 
     linear
    order
 \begin{align*}
 E_{\alpha_0} < E_{\delta+\alpha_0} <
  E_{2\delta+\alpha_0}
  < \cdots
  <
   E_{\delta} < E_{2\delta}
    < E_{3\delta}
 < \cdots
  <
   E_{2\delta + \alpha_1} <
     E_{\delta + \alpha_1} < E_{\alpha_1}.
     \end{align*}
   \end{proposition}
\noindent 
 We mention a variation on the formula  \eqref{eq:dam2introalt}.
By \cite[p.~307]{damiani} the following
holds  for $n\geq 1$:
\begin{equation}
\label{eq:Bdel2}
E_{n \delta} = 
q^{-2} W_1 E_{(n-1)\delta+\alpha_0} 
-  E_{(n-1)\delta+\alpha_0}  W_1. 
\end{equation}
\noindent 

\noindent 
Recall the antiautomorphism $\tau$ of $U^+_q$, from Definition \ref{def:tauA}.
\begin{lemma} \label{lem:asym2}  The antiautomorphism $\tau$  sends
$E_{n\delta+\alpha_0}\leftrightarrow E_{n\delta+\alpha_1}$ for $n \in \mathbb N$, and fixes $E_{n\delta}$ for $n\geq 1$.
\end{lemma}
\begin{proof} To verify the first assertion, compare the two relations in  \eqref{eq:dam1introalt}. To verify the second assertion, compare  \eqref{eq:dam2introalt} and
\eqref{eq:Bdel2}.
\end{proof}

\noindent  For the PBW basis in Proposition \ref{prop:PBWbasis}, the corresponding reduction rules were obtained by Damiani \cite[Section~4]{damiani}. These reduction rules are repeated below using adjusted notation.
\medskip

\noindent By \cite[p.~307]{damiani} the elements $\lbrace E_{n\delta}\rbrace_{n=1}^\infty$ mutually
commute.
\begin{lemma}
\label{lem:dam2}
{\rm (See \cite[p.~307]{damiani}.)}
For $i,j \in \mathbb N$ the following holds in $U_q^+$:
\begin{align*}
E_{i\delta+\alpha_1}
E_{j\delta+\alpha_0}
=
q^2
E_{j\delta+\alpha_0}
E_{i\delta+\alpha_1}
+
q^2 E_{(i+j+1)\delta}.
\end{align*}
\end{lemma}

\begin{lemma}
\label{lem:com3}
{\rm (See \cite[p.~300]{damiani}.)}
For  $i> j\geq 0$ the following hold in $U^+_q$.
\begin{enumerate}
\item[\rm (i)] Assume that $i-j=2r+1$ is odd. Then
\begin{align*}
&
E_{i\delta+\alpha_0}
E_{j\delta+\alpha_0}
=
q^{-2}
E_{j\delta+\alpha_0}
E_{i\delta+\alpha_0}
-
(q^2-q^{-2})\sum_{\ell=1}^{r}
q^{-2\ell}
E_{(j+\ell) \delta+\alpha_0}
E_{(i-\ell) \delta+\alpha_0},
\\
&
E_{j\delta+\alpha_1}
E_{i\delta+\alpha_1} =
q^{-2}
E_{i\delta+\alpha_1 }
E_{j\delta+\alpha_1 }
-
(q^2-q^{-2})\sum_{\ell=1}^{r}
q^{-2\ell}
E_{(i-\ell) \delta+\alpha_1}
E_{(j+\ell) \delta+\alpha_1}.
\end{align*}
\item[\rm (ii)] Assume that $i-j=2r$ is even. Then
\begin{align*}
E_{i\delta+\alpha_0}
E_{j\delta+\alpha_0}
 =
q^{-2}
E_{j\delta+\alpha_0}
&E_{i\delta+\alpha_0}
-
q^{j-i+1} (q-q^{-1}) E^2_{(r+j)\delta+\alpha_0}
\\
&-\;
(q^2-q^{-2})\sum_{\ell=1}^{r-1}
q^{-2\ell}
E_{(j+\ell) \delta+\alpha_0}
E_{(i-\ell) \delta+\alpha_0},
\\
E_{j\delta+\alpha_1}
E_{i\delta+\alpha_1} =
q^{-2}
E_{i\delta+\alpha_1 }
&
E_{j\delta+\alpha_1 }
-
q^{j-i+1} (q-q^{-1}) E^2_{(r+j)\delta+\alpha_1}
\\
&-\;
(q^2-q^{-2})\sum_{\ell=1}^{r-1}
q^{-2\ell}
E_{(i-\ell) \delta+\alpha_1}
E_{(j+\ell) \delta+\alpha_1}.
\end{align*}
\end{enumerate}
\end{lemma}

\begin{lemma}
\label{lem:com2}
{\rm (See \cite[p.~304]{damiani}.)}
For $i\geq 1$ and $j\geq 0$ the following hold in $U^+_q$:
\begin{align*}
&E_{i\delta} E_{j\delta+\alpha_0} =
E_{j\delta+ \alpha_0} E_{i\delta}
+ q^{2-2i}(q+q^{-1}) E_{(i+j)\delta+\alpha_0}
\\
& \qquad \qquad \qquad \qquad \qquad  \qquad -\;
q^2(q^2-q^{-2})\sum_{\ell=1}^{i-1}
q^{-2\ell}
E_{(j+\ell) \delta+\alpha_0}
E_{(i-\ell) \delta},
\\
&
E_{j\delta+\alpha_1}
E_{i\delta}
=
E_{i\delta}
E_{j\delta+ \alpha_1}
+ q^{2-2i}(q+q^{-1}) E_{(i+j)\delta+\alpha_1}
\\
& \qquad \qquad \qquad \qquad \qquad \qquad -\;
q^2(q^2-q^{-2})\sum_{\ell=1}^{i-1}
q^{-2\ell}
E_{(i-\ell) \delta}
E_{(j+\ell) \delta+\alpha_1}.
\end{align*}
\end{lemma}

\noindent We mention an alternative version of Lemma \ref{lem:com3}.

\begin{lemma}
\label{lem:dam1} {\rm (See \cite[Section~2.3]{charp} or \cite[Lemma~3.5]{catalan}.)} The following relations hold in $U^+_q$.
For $i \in \mathbb N$,
\begin{align*}
&
\lbrack E_{(i+1)\delta+\alpha_0},
E_{i\delta+\alpha_0}\rbrack_q=0, \qquad \qquad
\lbrack E_{i\delta+\alpha_1},
E_{(i+1)\delta+\alpha_1}\rbrack_q=0.
\end{align*}
\noindent For distinct $i,j \in \mathbb N$,
\begin{align*}
&
\lbrack E_{(i+1)\delta+\alpha_0},
E_{j\delta+\alpha_0}\rbrack_q +
\lbrack E_{(j+1)\delta+\alpha_0},
E_{i\delta+\alpha_0} \rbrack_q =0,
\\
&
\lbrack 
E_{j\delta+\alpha_1},
E_{(i+1)\delta+\alpha_1}\rbrack_q + 
\lbrack E_{i\delta+\alpha_1},
E_{(j+1)\delta+\alpha_1} \rbrack_q=0.
\end{align*}
\end{lemma}
\noindent We mention an alternative version of Lemma \ref{lem:com2}. For notational convenience define  
\begin{align*}
E_{0\delta}=-(q-q^{-1})^{-1}.
\end{align*}
\begin{lemma}
\label{lem:altxx} {\rm (See \cite[Lemma~3.4]{catalan}.)}
For $i,j\in \mathbb N$ the following hold in $U^+_q$:
\begin{align*}
&
\lbrack E_{i\delta+\alpha_0}, E_{(j+1)\delta} \rbrack = 
\lbrack E_{(i+1)\delta+\alpha_0}, E_{j\delta} \rbrack_{q^2},
\\
&
\lbrack E_{(j+1)\delta}, E_{i\delta+\alpha_1}\rbrack 
=
\lbrack E_{j\delta}, E_{(i+1)\delta+\alpha_1}\rbrack_{q^2}.
\end{align*}
\end{lemma}

\section{Generating functions for $U^+_q$}

\noindent In the previous section we displayed a PBW basis for $U^+_q$ along with the corresponding reduction rules. In this section we describe these reduction rules using generating functions.
We acknowledge that the material in this section is well known to the experts, and  readily follows from \cite[Section~IV]{DF} and \cite{beck, bcp}.
The material is included for use later in the paper.

\begin{definition} \label{def:Bgen}
\rm We define some generating functions in the indeterminate $t$:
\begin{align}
&E^-(t) = \sum_{n  \in \mathbb N}  E_{n\delta+\alpha_0} t^n, \qquad \quad
E^+(t) = \sum_{n  \in \mathbb N}  E_{n\delta+\alpha_1} t^n,
\\
&E(t) = \sum_{n\in \mathbb N} E_{n\delta} t^n. 
 \label{eq:zerodelta}
 \end{align}
\end{definition}
\noindent Observe that
\begin{align}
E^-(0)=W_0, \qquad \qquad E^+(0)=W_1, \qquad \qquad E(0)=-(q-q^{-1})^{-1}.
\label{eq:zv}
\end{align}

\begin{lemma}
\label{lem:BPhi}
For the algebra $U^+_q$,
\begin{align}
\label{eq:BP1}
&\frac{t \lbrack E_\delta, E^-(t) \rbrack}{q+q^{-1}}= 
E^-(t)-W_0,
\qquad \qquad
\frac{t\lbrack E^+(t), E_\delta\rbrack}{q+q^{-1}}= 
E^+(t)-W_1.
\end{align}
\end{lemma}
\begin{proof} These equations express the relations 
 \eqref{eq:dam1introalt} in terms of generating functions.
\end{proof}

  \begin{lemma} \label{prop:pbwRelP} 
  For the algebra $U^+_q$,
  \begin{align}
  \lbrack W_0, E^+(t) \rbrack_q &= 
  - q t^{-1} E(t) -\frac{qt^{-1}}{q-q^{-1}},    \label{eq:L1}
  \\
  \lbrack E^-( t), W_1 \rbrack_q &= 
  -  qt^{-1} E(t) -\frac{qt^{-1}}{q-q^{-1}}.        \label{eq:L2}
  \end{align}
  \end{lemma}
  \begin{proof} The equation \eqref{eq:L1} (resp.~\eqref{eq:L2}) expresses the relation  \eqref{eq:dam2introalt} (resp.~\eqref{eq:Bdel2}) in terms of generating
  functions.
  \end{proof}

\noindent For the rest of the paper, let $s$ denote an indeterminate that commutes with $t$. By the comment above Lemma \ref{lem:dam2},
  \begin{align}
  \label{eq:EsEt}
  \lbrack E(s), E(t)\rbrack=0.
  \end{align}
  
  \begin{proposition} \label{prop:New} For the algebra $U^+_q$,
  \begin{align}
  \lbrack E^-(s), E^+(t) \rbrack_q = -q \,\frac{E(s)-E(t)}{s-t}.
  \label{eq:New}
  \end{align}
  \end{proposition}
  \begin{proof} The equation \eqref{eq:New} expresses Lemma \ref{lem:dam2} in terms of generating functions.
  \end{proof}

   \begin{proposition}\label{prop:GFwang} For the algebra $U^+_q$,
  \begin{align}
  0&=\frac{qt-q^{-1} s}{q-q^{-1}}E^-(s) E^-(t) +\frac{qs-q^{-1}t}{q-q^{-1}}E^-(t) E^-(s)
 - s \bigl( E^-(s)\bigr)^2   - t \bigl( E^-(t)\bigr)^2,
  \label{eq:cc3}
\\
  0&= 
  \frac{qt-q^{-1} s}{q-q^{-1}}E^+(t) E^+(s) +\frac{qs-q^{-1}t}{q-q^{-1}}E^+(s) E^+(t)
  -s \bigl( E^+(s)\bigr)^2
  - t \bigl( E^+(t)\bigr)^2.
   \label{eq:cc1}
  \end{align}
  \end{proposition}
  \begin{proof} These equations express Lemma \ref{lem:dam1} 
  in terms of generating functions.
   \end{proof}
  
  \begin{proposition}
  \label{prop:wangGF}  For the algebra $U^+_q$,
  \begin{align}
 0 &=
  (s-q^2t)E^-(s)E(t) 
  +
  (q^{-2}t-s)E(t)E^-(s)
 + (q^2-q^{-2})t E^-(q^{-2}t) E(t),
 \label{eq:EE2}
 \\
  0 &= 
  (s-q^2t)E(t)E^+(s) 
  +
  (q^{-2}t-s)E^+(s)E(t)
 + (q^2-q^{-2})t E(t) E^+(q^{-2}t).
 \label{eq:EE1}
\end{align} 
\end{proposition}
\begin{proof} These equations express Lemma \ref{lem:altxx}
in terms of generating functions.
   \end{proof}

  \begin{corollary} \label{prop:pbwRel} 
  For the algebra $U^+_q$,
  \begin{align} 
  \label{eq:L7}
  \lbrack W_0, E^-( t) \rbrack_q &= 
  (q-q^{-1})\bigl( E^-(t )\bigr)^2,
  \\  \label{eq:L8}
  \lbrack W_0, E( t) \rbrack_{q^2} &= 
  (q^2-q^{-2}) E^-(q^{-2}t )E(t),
  \\ \label{eq:L9}
  \lbrack E^+( t), W_1\rbrack_q &= 
  (q-q^{-1})\bigl( E^+(t )\bigr)^2,
  \\ \label{eq:L10}
  \lbrack E( t), W_1 \rbrack_{q^2} &= 
  (q^2-q^{-2}) E(t) E^+(q^{-2}t ).
  \end{align}
  \end{corollary}
  \begin{proof} Set $s=0$  in  Propositions \ref{prop:GFwang}, \ref{prop:wangGF} and evaluate the results using \eqref{eq:zv}.
  \end{proof}

 \begin{remark}\label{rem:follow}
 \rm Lemmas 
\ref{lem:BPhi}, \ref{prop:pbwRelP} 
 and Corollary  \ref{prop:pbwRel} 
 follow from  \eqref{eq:zv}, 
 \eqref{eq:EsEt}
   and 
 Propositions   \ref{prop:New}, \ref{prop:GFwang}, \ref{prop:wangGF}. Indeed Lemma  \ref{prop:pbwRelP}  follows from  Proposition   \ref{prop:New} by setting $s=0$ or $t=0$, and evaluating the results using   \eqref{eq:zv}.
 Corollary  \ref{prop:pbwRel} follows from Propositions \ref{prop:GFwang}, \ref{prop:wangGF} by the proof of Corollary  \ref{prop:pbwRel}. Lemma \ref{lem:BPhi} follows from 
   \eqref{eq:L2},  \eqref{eq:L7},  \eqref{eq:L8}  along with
 \eqref{eq:EE2} at $s=t$.
 \end{remark}

 \section{The algebra $\mathcal U^+_q$}
   \noindent In the previous section we discussed the algebra $U^+_q$. In this section we discuss its alternating central extension $\mathcal U^+_q$.
  
\begin{definition}\rm
\label{def:Aq}
(See \cite[Definition~3.1]{altCE}.)
Define the algebra $\mathcal U^+_q$
by generators
\begin{align}
\label{eq:4gens}
\lbrace \mathcal W_{-k}\rbrace_{k\in \mathbb N}, \qquad  \lbrace \mathcal  W_{k+1}\rbrace_{k\in \mathbb N},\qquad  
 \lbrace \mathcal G_{k+1}\rbrace_{k\in \mathbb N},
\qquad
\lbrace \mathcal {\tilde G}_{k+1}\rbrace_{k\in \mathbb N}
\end{align}
 and the following relations. For $k, \ell \in \mathbb N$,
\begin{align}
&
 \lbrack \mathcal W_0, \mathcal W_{k+1}\rbrack= 
\lbrack \mathcal W_{-k}, \mathcal W_{1}\rbrack=
(1-q^{-2})({\mathcal{\tilde G}}_{k+1} - \mathcal G_{k+1}),
\label{eq:3p1}
\\
&
\lbrack \mathcal W_0, \mathcal G_{k+1}\rbrack_q= 
\lbrack {\mathcal{\tilde G}}_{k+1}, \mathcal W_{0}\rbrack_q= 
(q-q^{-1}) \mathcal W_{-k-1},
\label{eq:3p2}
\\
&
\lbrack \mathcal G_{k+1}, \mathcal W_{1}\rbrack_q= 
\lbrack \mathcal W_{1}, {\mathcal {\tilde G}}_{k+1}\rbrack_q= 
(q-q^{-1}) \mathcal W_{k+2},
\label{eq:3p3}
\\
&
\lbrack \mathcal W_{-k}, \mathcal W_{-\ell}\rbrack=0,  \qquad 
\lbrack \mathcal W_{k+1}, \mathcal W_{\ell+1}\rbrack= 0,
\label{eq:3p4}
\\
&
\lbrack \mathcal W_{-k}, \mathcal W_{\ell+1}\rbrack+
\lbrack \mathcal W_{k+1}, \mathcal W_{-\ell}\rbrack= 0,
\label{eq:3p5}
\\
&
\lbrack \mathcal W_{-k}, \mathcal G_{\ell+1}\rbrack+
\lbrack \mathcal G_{k+1}, \mathcal W_{-\ell}\rbrack= 0,
\label{eq:3p6}
\\
&
\lbrack \mathcal W_{-k}, {\mathcal {\tilde G}}_{\ell+1}\rbrack+
\lbrack {\mathcal {\tilde G}}_{k+1}, \mathcal W_{-\ell}\rbrack= 0,
\label{eq:3p7}
\\
&
\lbrack \mathcal W_{k+1}, \mathcal G_{\ell+1}\rbrack+
\lbrack \mathcal  G_{k+1}, \mathcal W_{\ell+1}\rbrack= 0,
\label{eq:3p8}
\\
&
\lbrack \mathcal W_{k+1}, {\mathcal {\tilde G}}_{\ell+1}\rbrack+
\lbrack {\mathcal {\tilde G}}_{k+1}, \mathcal W_{\ell+1}\rbrack= 0,
\label{eq:3p9}
\\
&
\lbrack \mathcal G_{k+1}, \mathcal G_{\ell+1}\rbrack=0,
\qquad 
\lbrack {\mathcal {\tilde G}}_{k+1}, {\mathcal {\tilde G}}_{\ell+1}\rbrack= 0,
\label{eq:3p10}
\\
&
\lbrack {\mathcal {\tilde G}}_{k+1}, \mathcal G_{\ell+1}\rbrack+
\lbrack \mathcal G_{k+1}, {\mathcal {\tilde G}}_{\ell+1}\rbrack= 0.
\label{eq:3p11}
\end{align}
The generators 
\eqref{eq:4gens} are called {\it alternating}. We call $\mathcal U^+_q$ the {\it alternating
central extension of $U^+_q$}. 
\noindent For notational convenience define
\begin{align}
{\mathcal G}_0 = 1, \qquad \qquad 
{\mathcal {\tilde G}}_0 = 1.
\label{eq:GG0}
\end{align}
\end{definition}

 \begin{remark}\rm
The relations in Definition \ref{def:Aq} resemble some relations
involving the $q$-Onsager algebra that were found earlier by Baseilhac and Shigechi
\cite[Definition~3.1]{basnc}; see also \cite{BK05}.
\end{remark}

\noindent Next we describe some symmetries of $\mathcal U^+_q$.
\begin{lemma}
\label{lem:autAc} {\rm (See \cite[Lemma~3.9]{altCE}.)} There exists an automorphism $\sigma$ of $\mathcal U^+_q$ that sends
\begin{align*}
\mathcal W_{-k} \mapsto \mathcal W_{k+1}, \qquad
\mathcal W_{k+1} \mapsto \mathcal W_{-k}, \qquad
\mathcal G_{k+1} \mapsto \mathcal {\tilde G}_{k+1}, \qquad
\mathcal {\tilde G}_{k+1} \mapsto \mathcal G_{k+1}
\end{align*}
 for $k \in \mathbb N$. Moreover $\sigma^2 = {\rm id}$.
\end{lemma}

\begin{lemma}\label{lem:antiautAc} {\rm (See \cite[Lemma~3.9]{altCE}.)} There exists an antiautomorphism $\dagger$ of $\mathcal U^+_q$ that sends
\begin{align*}
\mathcal W_{-k} \mapsto \mathcal W_{-k}, \qquad
\mathcal W_{k+1} \mapsto \mathcal W_{k+1}, \qquad
\mathcal G_{k+1} \mapsto \mathcal {\tilde G}_{k+1}, \qquad
\mathcal {\tilde G}_{k+1} \mapsto \mathcal G_{k+1}
\end{align*}
for $k \in \mathbb N$. Moreover $\dagger^2={\rm id}$.
\end{lemma}

\begin{lemma} \label{lem:sdcomAc}  
The maps $\sigma$, $\dagger $ commute.
\end{lemma}
\begin{proof} This is readily checked.
\end{proof}

\begin{definition}\label{def:tauAc}\rm Let $\tau$ denote the composition of the automorphism $\sigma$ from Lemma \ref{lem:autAc} and the antiautomorphism $\dagger$ from Lemma \ref{lem:antiautAc}. 
 Note that $\tau$ is an antiautomorphism of $\mathcal U^+_q$ that sends
\begin{align*}
\mathcal W_{-k} \mapsto \mathcal W_{k+1}, \qquad
\mathcal W_{k+1} \mapsto \mathcal W_{-k}, \qquad
\mathcal G_{k+1} \mapsto \mathcal G_{k+1}, \qquad
\mathcal {\tilde G}_{k+1} \mapsto \mathcal {\tilde G}_{k+1}
\end{align*}
\noindent for $k \in \mathbb N$. We have $\tau^2={\rm id}$.
\end{definition}

\noindent Next we discuss how $\mathcal U^+_q$ is related to $U^+_q$.

\begin{lemma}
\label{lem:iota}  {\rm (See \cite[Proposition~6.4]{altCE}.)}
There exists an algebra homomorphism $\imath: U^+_q \to \mathcal U^+_q$ that sends $W_0 \mapsto \mathcal W_0$ and
$W_1 \mapsto \mathcal W_1$. Moreover, $\imath$ is injective.
\end{lemma}

\begin{lemma} 
\label{lem:diag} 
The following diagrams commute:

\begin{equation*}
{\begin{CD}
U^+_q @>\imath  >> \mathcal U^+_q
              \\
         @V \sigma VV                   @VV \sigma V \\
         U^+_q @>>\imath >
                                 \mathcal U^+_q
                        \end{CD}}  	
         \qquad \qquad                
    {\begin{CD}
U^+_q @>\imath  >> \mathcal U^+_q
              \\
         @V \dagger VV                   @VV \dagger V \\
         U^+_q @>>\imath >
                                 \mathcal U^+_q
                        \end{CD}}  	     \qquad \qquad                
   {\begin{CD}
U^+_q @>\imath  >> \mathcal U^+_q
              \\
         @V \tau VV                   @VV \tau V \\
         U^+_q @>>\imath >
                                 \mathcal U^+_q
                        \end{CD}}  	                                       		    
\end{equation*}
\end{lemma}
\begin{proof} Chase the $U^+_q$-generators $W_0$, $W_1$ around each diagram, using
Lemmas  \ref{lem:aut}, \ref{lem:antiaut}
and Definition \ref{def:tauA}
along with Lemmas \ref{lem:autAc}, \ref{lem:antiautAc}
and Definition \ref{def:tauAc}.
\end{proof}

\section{Generating functions for $\mathcal U^+_q$}

\noindent In Definition \ref{def:Aq} the algebra $\mathcal U^+_q$ is defined by generators and relations.
In this section we describe the defining relations in terms of generating functions.

\begin{definition}
\label{def:gf4}
\rm (See \cite[Definition~A.1]{altCE}.)
We define some generating functions in the indeterminate $t$:
\begin{align*}
&\mathcal W^-(t) = \sum_{n \in \mathbb N} \mathcal W_{-n} t^n,
\qquad \qquad  
\mathcal W^+(t) = \sum_{n \in \mathbb N} \mathcal W_{n+1} t^n,
\\
&\mathcal G(t) = \sum_{n \in \mathbb N} \mathcal G_n t^n,
\qquad \qquad \qquad 
\mathcal {\tilde G}(t) = \sum_{n \in \mathbb N} \mathcal {\tilde G}_n t^n.
\end{align*}
\end{definition}
\noindent Observe that
\begin{align*}
\mathcal W^-(0) = \mathcal W_0, \qquad
\mathcal W^+(0) = \mathcal W_1, \qquad
\mathcal G(0) = 1, \qquad
\mathcal {\tilde G}(0) = 1.
\end{align*}

\noindent  We now give the relations \eqref{eq:3p1}--\eqref{eq:3p11}  in terms of generating functions. 

\begin{lemma} \label{lem:ad} {\rm (See \cite[Lemmas~A.2, A.3]{altCE}.)}
 For the algebra $\mathcal U^+_q$,
\begin{align}
& \label{eq:3pp1}
\lbrack \mathcal W_0, \mathcal W^+(t) \rbrack = \lbrack \mathcal W^-(t), \mathcal W_1 \rbrack = (1-q^{-2})t^{-1}(\mathcal {\tilde G}(t)-\mathcal G(t)),
\\
& \label{eq:3pp2}
\lbrack \mathcal W_0, \mathcal G(t) \rbrack_q = \lbrack \mathcal {\tilde G}(t), \mathcal W_0 \rbrack_q = (q-q^{-1}) \mathcal W^-(t),
\\
&\label{eq:3pp3}
\lbrack \mathcal G(t), \mathcal W_1 \rbrack_q = \lbrack \mathcal W_1, \mathcal {\tilde G}(t) \rbrack_q = (q-q^{-1}) \mathcal W^+(t),
\\
&\label{eq:3pp4}
\lbrack  \mathcal W^-(s), \mathcal W^-(t) \rbrack = 0, 
\qquad 
\lbrack \mathcal W^+(s),  \mathcal W^+(t) \rbrack = 0,
\\ \label{eq:3pp5}
&\lbrack  \mathcal W^-(s), \mathcal W^+(t) \rbrack 
+
\lbrack \mathcal W^+(s), \mathcal W^-(t) \rbrack = 0,
\\ \label{eq:3pp6}
&s \lbrack \mathcal W^-(s), \mathcal G(t) \rbrack 
+
t \lbrack  \mathcal G(s),  \mathcal W^-(t) \rbrack = 0,
\\ \label{eq:3pp7}
&s \lbrack  \mathcal W^-(s), \mathcal {\tilde G}(t) \rbrack 
+
t \lbrack  \mathcal {\tilde G}(s), \mathcal W^-(t) \rbrack = 0,
\\ \label{eq:3pp8}
&s \lbrack   \mathcal W^+(s),  \mathcal G(t) \rbrack
+
t \lbrack   \mathcal G(s), \mathcal W^+(t) \rbrack = 0,
\\ \label{eq:3pp9}
&s \lbrack   \mathcal W^+(s), \mathcal {\tilde G}(t) \rbrack
+
t \lbrack \mathcal {\tilde G}(s), \mathcal W^+(t) \rbrack = 0,
\\ \label{eq:3pp10}
&\lbrack   \mathcal G(s), \mathcal G(t) \rbrack = 0, 
\qquad 
\lbrack  \mathcal {\tilde G}(s),  \mathcal {\tilde G}(t) \rbrack = 0,
\\ \label{eq:3pp11}
&\lbrack  \mathcal {\tilde G}(s), \mathcal G(t) \rbrack +
\lbrack   \mathcal G(s), \mathcal {\tilde G}(t) \rbrack = 0.
\end{align}
\end{lemma}

\section{ A PBW basis for $\mathcal U^+_q$}

\noindent In \cite[Theorem~10.2]{altCE}  a PBW basis for $\mathcal U^+_q$ is obtained from the alternating generators in a certain linear order; see 
Appendix A below.
In the present section
we modify the  linear order to get a new PBW basis for $\mathcal U^+_q$ that is better suited to our purpose.
  For the new PBW basis we display the corresponding reduction rules.


\begin{definition}
\label{def:2p} \rm
Let $L$ denote the subalgebra of $\mathcal U^+_q$ generated by $\lbrace \mathcal W_{-k}\rbrace_{k\in \mathbb N}$, $\lbrace \mathcal G_{k+1}\rbrace_{k\in \mathbb N}$.
Let $R$ denote the subalgebra of $\mathcal U^+_q$ generated by $\lbrace \mathcal W_{k+1}\rbrace_{k\in \mathbb N}$, $\lbrace \mathcal {\tilde G}_{k+1}\rbrace_{k\in \mathbb N}$.
\end{definition}

\begin{lemma} \label{lem:XPBW} The following {\rm (i)--(iii)} hold for the subalgebras $L$ and $R$:
\begin{enumerate}
\item[\rm (i)]
a PBW basis for $L$ is obtained by the elements $\lbrace \mathcal W_{-i}\rbrace_{i\in \mathbb N}$, $\lbrace \mathcal G_{j+1}\rbrace_{j\in \mathbb N}$
in any linear order such that $\mathcal W_{-i} < \mathcal G_{j+1}$ for  $i,j \in \mathbb N$;
\item[\rm (ii)]  a PBW basis for $R$ is obtained by the elements $\lbrace \mathcal {\tilde G}_{k+1}\rbrace_{k\in \mathbb N}$, $\lbrace \mathcal W_{\ell+1}\rbrace_{\ell\in \mathbb N}$
in any linear order such that
$\mathcal {\tilde G}_{k+1} < \mathcal W_{\ell+1}$ for  $k, \ell \in \mathbb N$;
\item[\rm (iii)] the multiplication map
\begin{align*}
L
\otimes
R
 & \to   \mathcal U^+_q
\\
 l \otimes r  &\mapsto  l r
 \end{align*}
is an isomorphism of vector spaces.
\end{enumerate}
\end{lemma}
\begin{proof} We refer to Appendix A.
\\
\noindent (i) By Lemma \ref{lem:pbwP}
and the third displayed equation in Lemma \ref{lem:rr1}.
\\
\noindent (ii) By Lemma \ref{lem:pbwP}
and the last displayed equation in Lemma \ref{lem:rr1}.
\\
\noindent (iii)  By Lemma \ref{lem:pbwP} and (i), (ii) above.
\end{proof}
\noindent Recall from Lemma \ref{lem:autAc} the automorphism $\sigma$ of $\mathcal U^+_q$.

\begin{lemma} \label{lem:swap} The automorphism $\sigma$ sends $L \leftrightarrow R$.
\end{lemma}
\begin{proof} By Lemma \ref{lem:autAc} and Definition \ref{def:2p}.
\end{proof}

\begin{lemma} \label{lem:XPBW2} The following {\rm (i), (ii)} hold for the subalgebras $L$ and $R$:
\begin{enumerate}
\item[\rm (i)]
a PBW basis for $L$ is obtained by the elements $\lbrace \mathcal G_{i+1}\rbrace_{i\in \mathbb N}$, $\lbrace \mathcal W_{-j}\rbrace_{j\in \mathbb N}$
in any linear order such that $\mathcal G_{i+1} < \mathcal W_{-j}$ for  $i,j \in \mathbb N$;
\item[\rm (ii)]  a PBW basis for $R$ is obtained by the elements $\lbrace \mathcal W_{k+1}\rbrace_{k\in \mathbb N}$, $\lbrace \mathcal {\tilde G}_{\ell+1}\rbrace_{\ell\in \mathbb N}$
in any linear order such that
$\mathcal W_{k+1} < \mathcal {\tilde G}_{\ell+1}$ for  $k, \ell \in \mathbb N$.
\end{enumerate}
\end{lemma}
\begin{proof} (i) Apply $\sigma$ to the PBW basis for $R$ given in Lemma \ref{lem:XPBW}(ii), and use Lemmas  \ref{lem:autAc}, \ref{lem:swap}.
\\
\noindent (ii)  Apply $\sigma$ to the PBW basis for $L$ given in Lemma \ref{lem:XPBW}(i), and use Lemmas  \ref{lem:autAc}, \ref{lem:swap}.
\end{proof}

\begin{proposition} \label{lem:pbw}   A PBW basis for $\mathcal U^+_q$ is obtained by its alternating generators in any linear order $<$ such that
\begin{align}
\mathcal G_{i+1} < \mathcal W_{-j} < \mathcal W_{k+1} < \mathcal {\tilde G}_{\ell+1}\qquad \qquad i,j,k, \ell \in \mathbb N.
\label{eq:order}
\end{align}
\end{proposition} 
\begin{proof} By Lemma \ref{lem:XPBW}(iii) and Lemma \ref{lem:XPBW2}.
\end{proof}

\noindent For the PBW basis in Proposition \ref{lem:pbw}, the nontrivial reduction rules are a consequence of the following result.

\begin{lemma}
\label{lem:redrel2} 
For the algebra $\mathcal U^+_q$ we have
\begin{align*}
\mathcal W^+(s) \mathcal W^-(t) &= 
\mathcal W^-(t) \mathcal W^+(s) +
(1-q^{-2})\frac{\mathcal G(s) \mathcal {\tilde G}(t)-\mathcal G(t)
\mathcal {\tilde G}(s)}{s-t},
\\
\mathcal {\tilde G}(s) \mathcal G(t) &= 
\mathcal G(t) \mathcal {\tilde G}(s) +
(1-q^{2})st \frac{\mathcal W^-(t) \mathcal W^{+}(s)-\mathcal W^-(s)
\mathcal W^+(t)}{s-t}
\end{align*}
\noindent and also
\begin{align*}
\mathcal W^-(s) \mathcal G(t) &= q^{-1}\frac{(qs-q^{-1}t) \mathcal G(t) \mathcal W^-(s)-(q-q^{-1})t \mathcal G(s)\mathcal W^-(t)}{s-t},
\\
\mathcal W^+(s) \mathcal G(t) &= 
q \frac{(q^{-1}s-qt)\mathcal G(t) \mathcal W^+(s) + (q-q^{-1})t 
\mathcal G(s) \mathcal W^+(t)}{s-t},
\\
\mathcal {\tilde G}(s) \mathcal W^-(t) &= 
q^{-1} \frac{(q^{-1}s-qt)\mathcal W^-(t) \mathcal {\tilde G}(s) + (q-q^{-1})s 
\mathcal W^-(s) \mathcal {\tilde G}(t)}{s-t},
\\
\mathcal {\tilde G}(s) \mathcal W^+(t) &=  q
\frac{(qs-q^{-1}t) \mathcal W^+(t)\mathcal {\tilde G}(s)-(q-q^{-1})s \mathcal W^+(s) \mathcal {\tilde G}(t)}{s-t}.
\end{align*}
\end{lemma}
\begin{proof} For each of the above equations, either the equation or its $\sigma$-image is listed in Lemma \ref{lem:redrel2A}.
\end{proof}

\noindent Next we give the nontrivial reduction rules for the PBW basis in Proposition \ref{lem:pbw}.

 \begin{proposition} \label{prop:rr3} The following relations hold in $\mathcal U^+_q$. For $i,j \in \mathbb N$,
 \begin{align*}
 \mathcal W_{i+1} \mathcal W_{-j} &=\mathcal W_{-j} \mathcal W_{i+1} 
 + q^{-1}(q-q^{-1})\sum_{\ell=0}^{{\rm min}(i,j)} \bigl( \mathcal G_{i+j+1-\ell} \mathcal {\tilde G}_\ell - \mathcal G_\ell \mathcal {\tilde G}_{i+j+1-\ell} \bigr),
 \\
 \mathcal {\tilde G}_{i+1} \mathcal G_{j+1} &= \mathcal G_{j+1} \mathcal {\tilde G}_{i+1} +
 q(q-q^{-1}) \sum_{\ell=0}^{{\rm min}(i,j)} \bigl( 
 \mathcal W_{\ell-i-j-1} \mathcal W_{\ell+1} - \mathcal W_{-\ell} \mathcal W_{i+j+2-\ell} \bigr),
 \end{align*}
 and also
 \begin{align*}
 \mathcal W_{-i} \mathcal G_{j+1} &= \mathcal G_{j+1} \mathcal W_{-i} 
 + 
 q^{-1}(q-q^{-1}) \sum_{\ell=0}^{{\rm min}(i,j)} \bigl( \mathcal G_{\ell} \mathcal W_{\ell-i-j-1} - \mathcal G_{i+j+1-\ell} \mathcal W_{-\ell}\bigr),
 \\
 \mathcal W_{i+1} \mathcal G_{j+1} &= \mathcal G_{j+1} \mathcal W_{i+1} 
 + 
 q(q-q^{-1}) \sum_{\ell=0}^{{\rm min}(i,j)} \bigl( \mathcal G_{i+j+1-\ell} \mathcal W_{\ell+1} - \mathcal G_\ell \mathcal W_{i+j+2-\ell}\bigr),
 \\
  \mathcal {\tilde G}_{i+1} \mathcal W_{-j} &= \mathcal W_{-j} \mathcal {\tilde G}_{i+1} +
 q^{-1}(q-q^{-1}) \sum_{\ell=0}^{{\min}(i,j)} \bigl( 
 \mathcal W_{\ell-i-j-1} \mathcal {\tilde G}_{\ell} - \mathcal W_{-\ell} \mathcal {\tilde G}_{i+j+1-\ell}\bigr),
 \\
   \mathcal {\tilde G}_{i+1} \mathcal W_{j+1} &= \mathcal W_{j+1} \mathcal {\tilde G}_{i+1} +
 q(q-q^{-1}) \sum_{\ell=0}^{{\min}(i,j)} \bigl( 
 \mathcal W_{\ell+1} \mathcal {\tilde G}_{i+j+1-\ell} - \mathcal W_{i+j+2-\ell} \mathcal {\tilde G}_{\ell}\bigr).
 \end{align*}
 \end{proposition}
 \begin{proof} These relations are obtained by unpacking the equations in Lemma \ref{lem:redrel2}.
 \end{proof}


\section{The center of $\mathcal U^+_q$}
Earlier in this paper we discussed the generating functions $E^\pm(t)$, $E(t)$ for $U^+_q$
and the generating functions $\mathcal W^\pm(t)$, $\mathcal G(t)$, $\mathcal {\tilde G}(t)$ for $\mathcal U^+_q$. In the next section, we will investigate how
$E^\pm(t)$, $E(t)$ are related to 
 $\mathcal W^\pm(t)$, $\mathcal G(t)$, $\mathcal {\tilde G}(t)$. In the present section, 
we prepare for this investigation with some remarks about the center 
$\mathcal Z$ of $\mathcal U^+_q$. 

\begin{definition} 
\label{def:Zvee}
\rm (See \cite[Definition~5.1]{altCE}.)
For $n \geq 1 $ define
\begin{align}
 \mathcal Z^\vee_n= \sum_{k=0}^n \mathcal G_k \mathcal {\tilde G}_{n-k} q^{n-2k}
-  q \sum_{k=0}^{n-1} \mathcal W_{-k} \mathcal W_{n-k} q^{n-1-2k}.
\label{eq:ZnVee}
\end{align}
For notational convenience define $\mathcal Z^\vee_0=1$.
\end{definition}

\noindent Next, we interpret Definition \ref{def:Zvee} in terms of generating functions.
\begin{definition}\label{def:Zn} Define the generating function
 \begin{align*}
 \mathcal Z^\vee(t) = \sum_{n\in \mathbb N} \mathcal Z^\vee_n  t^n.
 \end{align*}
 \end{definition}

\begin{lemma} \label{lem:zV1} {\rm (See \cite[Lemma~A.8]{altCE}.)}
We have
\begin{align}
\mathcal Z^\vee(t) &= 
 \mathcal G(q^{-1}t) \mathcal{\tilde G}(qt) -qt
\mathcal W^{-} (q^{-1}t)\mathcal W^+(qt).
\label{lem:zs}
\end{align}
\end{lemma}

\begin{lemma} \label{lem:Zfix} {\rm (See \cite[Lemma~5.2 and Proposition~8.3]{altCE}.)}
For $n\geq 1$ we have $\mathcal Z^\vee_n\in \mathcal Z$. Moreover $\mathcal Z^\vee_n$  fixed by $\sigma$ and $\dagger$ and $\tau$.
\end{lemma}

\begin{definition}\rm Let $\langle \mathcal W_0, \mathcal W_1 \rangle$ denote the subalgebra of $\mathcal U^+_q$ generated by $\mathcal W_0$, $\mathcal W_1$.
\end{definition}

\begin{proposition} \label{lem:sum}
{\rm (See \cite[Section~6]{altCE}.)} For the algebra $\mathcal U^+_q$ the following {\rm (i)--(iii)} hold:
\begin{enumerate}
 \item[\rm (i)] 
there exists an algebra isomorphism
$U^+_q \to \langle \mathcal W_0, \mathcal W_1\rangle$ 
that sends $W_0\mapsto \mathcal W_0$ and
$W_1\mapsto \mathcal W_1$;
\item[\rm (ii)] 
there exists an algebra isomorphism
 $\mathbb F\lbrack z_1, z_2,\ldots \rbrack \to \mathcal Z$
 that sends $z_n \mapsto \mathcal Z^\vee_n$ for $n \geq 1$;
\item[\rm (iii)] 
the multiplication map 
\begin{align*}
\langle \mathcal W_0, \mathcal W_1\rangle 
\otimes
\mathcal Z
 & \to   \mathcal U^+_q
\\
 w \otimes z  &\mapsto  wz
 \end{align*}
 is an isomorphism of algebras.
 \end{enumerate}
\end{proposition}

\noindent Note that the isomorphism in Proposition \ref{lem:sum}(i) is induced by the map $\imath$ from Lemma \ref{lem:iota}.


\noindent We emphasize a few points.
\begin{corollary}\label{cor:sum} For the algebra $\mathcal U^+_q$
 the following {\rm (i)--(iii)} hold:
 \begin{enumerate}
 \item[\rm (i)] the algebra $\mathcal U^+_q$ is generated by $\mathcal W_0$, $\mathcal W_1$, $\mathcal Z$;
 \item[\rm (ii)] the elements $\lbrace \mathcal Z^\vee_n \rbrace_{n=1}^\infty$ are algebraically independent and generate $\mathcal Z$;
 \item[\rm (iii)] everything in $\mathcal Z$ is fixed by $\sigma$ and $\dagger$ and $\tau$.
 \end{enumerate}
 \end{corollary}
 \begin{proof} (i) By Proposition  \ref{lem:sum}(iii). \\
 \noindent (ii) By Proposition  \ref{lem:sum}(ii). \\
 \noindent (iii) By (ii) above and Lemma  \ref{lem:Zfix}.
 \end{proof}

\section{Comparing the generating functions for $U^+_q$ and $\mathcal U^+_q$}
\noindent In this section we investigate how the generating functions $E^\pm(t)$, $E(t)$ for $U^+_q$
are related to the generating functions $\mathcal W^\pm(t)$, $\mathcal G(t)$, $\mathcal {\tilde G}(t)$ for $\mathcal U^+_q$.
\medskip

\noindent Throughout this section, we identify $U^+_q$ with $\langle \mathcal W_0, \mathcal W_1 \rangle $ via the map $\imath$ from 
Lemma \ref{lem:iota}. For notational convenience define
\begin{align}
\xi = -q^2(q-q^{-1})^{-2}.
\end{align}

\begin{proposition} \label{lem:zzznote}
For the algebra $\mathcal U^+_q$,
\begin{align}
\label{eq:long1}
\mathcal W^-(t) = E^-(q \xi t) \mathcal {\tilde G}(t) = \mathcal {\tilde G}(t) E^-(q^{-1} \xi t),
\\
\mathcal W^+(t) = E^+(q^{-1} \xi t) \mathcal {\tilde G}(t)= \mathcal {\tilde G}(t)E^+(q \xi t).
\label{eq:long2}
\end{align}
\end{proposition}
\begin{proof} The equation on the left in \eqref{eq:long1} (resp. \eqref{eq:long2}) is equation (15) (resp. equation (14)) in \cite{compactUqp}, expressed in terms of generating functions.
Using the antiautomorphism $\tau$ we get the equations on the right in \eqref{eq:long1}, \eqref{eq:long2}.
\end{proof}

\noindent In the next two results we give some consequences of Proposition  \ref{lem:zzznote}.

\begin{proposition}\label{prop:extra1} For the algebra $\mathcal U^+_q$,
\begin{align}
&\mathcal {\tilde G}(t)  \mathcal W_0 = \Bigl(q^{-2} \mathcal W_0 + q^{-1}(q-q^{-1}) E^-(q\xi t)\Bigr) \mathcal {\tilde G}(t), \label{eq:four1}
\\
&\mathcal {\tilde G}(t)  \mathcal W_1 = \Bigl(q^{2} \mathcal W_1 - q(q-q^{-1}) E^+(q^{-1}\xi t)\Bigr) \mathcal {\tilde G}(t) \label{eq:four2}
\end{align}
and also
\begin{align}
&\mathcal W_0 \mathcal {\tilde G}(t) = \mathcal {\tilde G}(t) \Bigl(q^{2} \mathcal W_0 - q(q-q^{-1}) E^-(q^{-1}\xi t)\Bigr), \label{eq:four3}
\\
&\mathcal W_1 \mathcal {\tilde G}(t)  = \mathcal {\tilde G}(t) \Bigl(q^{-2} \mathcal W_1 + q^{-1}(q-q^{-1}) E^+(q\xi t)\Bigr). \label{eq:four4}
\end{align}
\end{proposition}
\begin{proof} By \eqref{eq:3pp2}, \eqref{eq:3pp3}  we have
\begin{align*}
\lbrack \mathcal {\tilde G}(t), \mathcal W_0 \rbrack_q = (q-q^{-1}) \mathcal W^-(t), \qquad \qquad 
\lbrack \mathcal W_1, \mathcal {\tilde G}(t) \rbrack_q = (q-q^{-1}) \mathcal W^+(t).
\end{align*}
In these equations, eliminate  $\mathcal W^-(t)$ and  $\mathcal W^+(t)$ using Proposition  \ref{lem:zzznote}, and simplify the result.
\end{proof}

\begin{proposition} \label{lem:ex2}
For the algebra $\mathcal U^+_q$,
\begin{align}
\label{eq:GF1}
 \mathcal G(t)&=\Bigl(q^2 t E^-(q\xi t) E^+(q^{-1}\xi t) -(q-q^{-1}) E (q \xi t) \Bigr) \mathcal {\tilde G}(t)
 \\
 &=\Bigl( t E^+(q^{-1}\xi t) E^-(q\xi t) -(q-q^{-1}) E (q^{-1} \xi t) \Bigr) \mathcal {\tilde G}(t) \label{eq:GF2}
 \\
  &=\mathcal {\tilde G}(t) \Bigl(q^2 t E^-(q^{-1} \xi t) E^+(q\xi t) -(q-q^{-1}) E (q \xi t) \Bigr) \label{eq:GF3}
 \\
 &= \mathcal {\tilde G}(t) \Bigl( t E^+(q\xi t) E^-(q^{-1}\xi t) -(q-q^{-1}) E (q^{-1} \xi t) \Bigr). \label{eq:GF4}
\end{align}
\end{proposition}
\begin{proof} We first show \eqref{eq:GF1}. By \eqref{eq:3pp1},
\begin{align*}
\mathcal W^-(t) \mathcal W_1 - \mathcal W_1 \mathcal W^-(t) = (1-q^{-2})t^{-1} (\mathcal {\tilde G}(t)-\mathcal G(t)).
\end{align*}
In this equation, eliminate $\mathcal W^-(t)$ using the equation on the left in \eqref{eq:long1}. Evaluate the resulting equation using \eqref{eq:four2}.
In the resulting equation, eliminate $\lbrack E^-(q \xi t), \mathcal W_1 \rbrack_q$ using   \eqref{eq:L2}. The resulting
equation becomes \eqref{eq:GF1} after simplification.
We have shown \eqref{eq:GF1}. 
The right-hand sides of \eqref{eq:GF1}, \eqref{eq:GF2} are equal by
Proposition \ref{prop:New}, so \eqref{eq:GF2} holds. Using $\tau$ we obtain
\eqref{eq:GF3}, \eqref{eq:GF4}.
\end{proof}
\begin{remark}\rm The above Propositions  \ref{lem:zzznote},  \ref{lem:ex2} are variations on 
\cite[Proposition~5.18]{basFMA} and \cite[Proposition~5.20]{basFMA}.
\end{remark}

\noindent We have a comment.   The generating function $\mathcal {\tilde G}(t)$ is invertible
by Lemma \ref{lem:inv}
and $\mathcal {\tilde G}_0=1$.
\begin{lemma}\label{lem:gi}
For the algebra $\mathcal U^+_q$,
\begin{align}
\bigl(\mathcal {\tilde G}(t)\bigr)^{-1} \mathcal W^-(t) &= \mathcal W^-(q^{-2}t) \bigl(\mathcal {\tilde G}(q^{-2}t)\bigr)^{-1}, \label{eq:gi1}
\\
\bigl(\mathcal {\tilde G}(t)\bigr)^{-1} \mathcal W^+(t) &=
 \mathcal W^+(q^2t) \bigl(\mathcal {\tilde G}(q^2t)\bigr)^{-1}. \label{eq:gi2}
\end{align}
\end{lemma}
\begin{proof} To get \eqref{eq:gi1}, compare the two equations in \eqref{eq:long1}.
To get \eqref{eq:gi2}, compare the two equations in \eqref{eq:long2}.
\end{proof}

\section{A factorization of $\mathcal Z^\vee(t)$}


\noindent Recall the generating function $\mathcal Z^\vee(t)$ from  
Definition \ref{def:Zn} and Lemma  \ref{lem:zV1}.
In this section we obtain a factorization of $\mathcal Z^\vee(t)$.
\medskip

\noindent
 Throughout this section we identify $U^+_q$ with
$\langle \mathcal W_0, \mathcal W_1\rangle $ via the map $\imath$ from Lemma
\ref{lem:iota}.

\begin{proposition}  \label{thm:Mn} 
For the algebra $\mathcal U^+_q$ we have
\begin{align}
\mathcal Z^\vee(t)  =-(q-q^{-1}) \mathcal {\tilde G}(q^{-1}t) E(\xi t) \mathcal {\tilde G}(qt),
\label{eq:mainRes}
\end{align}
\noindent where we recall $
\xi = -q^2 (q-q^{-1})^{-2}$.
\end{proposition} 
\begin{proof} Consider the terms on the right in \eqref{lem:zs}.
By \eqref{eq:GF3},
\begin{align}
 \mathcal G(q^{-1}t) = 
 \mathcal {\tilde G}(q^{-1} t)
 \Bigl(qt E^-(q^{-2} \xi t)E^+(\xi t) - (q-q^{-1}) E(\xi t)\Bigr).
 \label{eq:s1}
 \end{align}
By Proposition  \ref{lem:zzznote},
\begin{align}
\mathcal W^-(q^{-1}t) =
\mathcal {\tilde G}(q^{-1}t)
 E^-(q^{-2}\xi t), \qquad \qquad
\mathcal W^+(qt)   = E^+(\xi t) \mathcal {\tilde G}(qt).
\label{eq:s2}
\end{align}
Evaluating the right-hand side of \eqref{lem:zs} using \eqref{eq:s1}, \eqref{eq:s2}  we routinely obtain \eqref{eq:mainRes}.
\end{proof}

\noindent Next, we give some consequences of Proposition \ref{thm:Mn}.

\begin{definition}\rm For notational convenience, define
\begin{align}
E^\vee(t) = -(q-q^{-1})E(t).
\label{eq:nc}
\end{align}
\end{definition}

\begin{corollary} \label{cor:BZ} For the algebra $\mathcal U^+_q$ we have
\begin{align}
\label{eq:BZ}
E^\vee (t) = 
\bigl(\mathcal {\tilde G}(q^{-1}\xi^{-1} t)\bigr)^{-1} \mathcal Z^\vee(\xi^{-1}t)
 \bigl(\mathcal {\tilde G}(q\xi^{-1}t)\bigr)^{-1}.
 \end{align}
 \end{corollary}
 \begin{proof} Rearrange the terms in \eqref{eq:mainRes}.
 \end{proof}
 \begin{corollary} \label{cor:BG}
 For the algebra $\mathcal U^+_q$,
 \begin{align*}
 \lbrack \mathcal {\tilde G}(s), E^\vee(t) \rbrack = 0.
 \end{align*}
 \end{corollary}
 \begin{proof} The generating function $\mathcal {\tilde G}(s)$ commutes with each factor on the right in \eqref{eq:BZ}.
 \end{proof}
 \begin{corollary} \label{cor:BGc}
 For the algebra $\mathcal U^+_q$,
 \begin{align*}
 \lbrack \mathcal {\tilde G}_{k+1},  E_{n \delta} \rbrack = 0 \qquad \qquad k,n \in \mathbb N.
 \end{align*}
 \end{corollary}
 \begin{proof} By Corollary 
 \ref{cor:BG}.
 \end{proof}
 \begin{corollary} \label{cor:perm} The generating function
  $\mathcal Z^\vee (t)$
 \noindent is equal to each of
 \begin{align*}
 \mathcal {\tilde G}(q^{-1}t) E^\vee(\xi t) \mathcal {\tilde G}(qt),
 \qquad \quad
  E^\vee(\xi t) \mathcal {\tilde G}(q^{-1}t) \mathcal {\tilde G}(qt),
 \qquad \quad
  \mathcal {\tilde G}(q^{-1}t) \mathcal {\tilde G}(qt) E^\vee(\xi t),
 \\
  \mathcal {\tilde G}(qt) E^\vee(\xi t) \mathcal {\tilde G}(q^{-1}t),
 \qquad \quad
  E^\vee(\xi t) \mathcal {\tilde G}(qt) \mathcal {\tilde G}(q^{-1}t),
 \qquad \quad
\mathcal {\tilde G}(qt) \mathcal {\tilde G}(q^{-1}t) E^\vee(\xi t).
 \end{align*}
 \end{corollary}
 \begin{proof} Evaluate \eqref{eq:mainRes}
 using \eqref{eq:nc} along with
  Corollary \ref{cor:BGc} and the equation on the right in \eqref{eq:3p10}.
 \end{proof}

 \section{Expressing $E^\pm(t)$, $E(t)$ in terms of $\mathcal W^\pm (t)$, $\mathcal G(t)$, $\mathcal {\tilde G}(t)$}
 
 \noindent In this section, we continue to discuss the generating functions $E^\pm(t)$, $E(t)$ for $U^+_q$ and $\mathcal W^\pm (t)$, $\mathcal G(t)$, $\mathcal {\tilde G}(t)$ for $\mathcal U^+_q$.
 We first express $E^\pm(t)$, $E(t)$ in terms of $\mathcal W^\pm (t)$, $\mathcal G(t)$, $\mathcal {\tilde G}(t)$. We then use these
 expressions to recover the results about  $E^\pm(t)$, $E(t)$ from Section 5. 
 \medskip
 
 \noindent Throughout this section we identify $U^+_q$ with
$\langle \mathcal W_0, \mathcal W_1\rangle $ via the map $\imath$ from Lemma
\ref{lem:iota}.
\medskip
 
\noindent To simplify our calculations, we use the following change of variables involving $\mathcal G(t)$, $\mathcal Z^\vee(t)$.
    \begin{lemma} \label{lem:sw}
 For the algebra $\mathcal U^+_q$,
 \begin{align*}
 \mathcal G(t) &= \mathcal Z^\vee(qt) \bigl( \mathcal {\tilde G}(q^2 t) \bigr)^{-1} 
 + q^2 t \mathcal W^-(t)  \mathcal W^+(q^2t)
 \bigl( \mathcal {\tilde G}(q^2t) \bigr)^{-1}.
\end{align*}
\end{lemma} 
\begin{proof} Solve \eqref{lem:zs} for
$\mathcal G(t)$.
\end{proof} 
 
\begin{theorem} \label{prop:back1}
For the algebra $\mathcal U^+_q$,
\begin{align}
E^-(t) &= \mathcal W^-(q^{-1} \xi^{-1} t) \bigl( \mathcal {\tilde G}(q^{-1} \xi^{-1} t)\bigr)^{-1}, \label{eq:em}
 \\
E^+(t) &=    \mathcal W^+(q \xi^{-1} t) \bigl( \mathcal {\tilde G}(q \xi^{-1} t)\bigr)^{-1}, \label{eq:ep}
\\
E(t) &= -\,\frac{\mathcal Z^\vee(\xi^{-1}t) \bigl( \mathcal {\tilde G}(q^{-1} \xi^{-1} t)\bigr)^{-1} \bigl( \mathcal {\tilde G}(q \xi^{-1} t)\bigr)^{-1}}{q-q^{-1}}.
\label{eq:e}
 \end{align}
 \end{theorem}
 \begin{proof} To get \eqref{eq:em},  replace $t$ by $q^{-1}\xi^{-1}t$ in the equation on the left in \eqref{eq:long1}.
 To get \eqref{eq:ep}, replace $t$ by $q\xi^{-1}t$ in the equation on the left in \eqref{eq:long2}.
 To get \eqref{eq:e}, replace $t$ by $ \xi^{-1} t$ in Corollary \ref{cor:perm}.
  \end{proof}

\noindent In Section 5, we gave some
relations involving $E^{\pm} (t)$, $E(t)$. Our next goal is to recover these relations using Theorem \ref{prop:back1}.
In order to  make use of Theorem \ref{prop:back1}, we display some equations involving $\bigl(\mathcal {\tilde G}(t)\bigr)^{-1}$.

\begin{proposition} 
\label{thm:mgi}
For the algebra $\mathcal U^+_q$,
\begin{align}
&\quad \bigl(\mathcal {\tilde G}(s)\bigr)^{-1} \mathcal {\tilde G}(t) = \mathcal {\tilde G}(t) \bigl(\mathcal {\tilde G}(s)\bigr)^{-1}, \qquad
 \bigl(\mathcal {\tilde G}(s)\bigr)^{-1} \bigl( \mathcal {\tilde G}(t)\bigr)^{-1} = \bigl(\mathcal {\tilde G}(t) \bigr)^{-1}\bigl(\mathcal {\tilde G}(s)\bigr)^{-1}, \label{eq:G1}
\\
\begin{split}
&\quad\bigl(\mathcal {\tilde G}(s)\bigr)^{-1} \mathcal W^-(t) =\\
&\frac{q(s-t)  \mathcal W^-(t) \bigl(\mathcal {\tilde G}(s)\bigr)^{-1}- (q-q^{-1}) s  \mathcal W^-(q^{-2}s) \bigl(\mathcal {\tilde G}(s)\bigr)^{-1} \bigl(\mathcal {\tilde G}(q^{-2}s)\bigr)^{-1} \mathcal {\tilde G}(t)}{q^{-1}s-qt},
\end{split}\label{eq:G2}
\\
 \begin{split}
&\quad    \bigl(\mathcal {\tilde G}(s)\bigr)^{-1} \mathcal W^+(t) =\\
&\frac{q^{-1}(s-t)  \mathcal W^+(t)\bigl(\mathcal {\tilde G}(s)\bigr)^{-1}  +(q-q^{-1}) s   \mathcal W^+(q^{2}s)  \bigl(\mathcal {\tilde G}(s)\bigr)^{-1} \bigl(\mathcal {\tilde G}(q^{2}s)\bigr)^{-1} \mathcal {\tilde G}(t)  }{qs-q^{-1}t},
\end{split}\label{eq:G3}
\\
&\quad \bigl(\mathcal {\tilde G}(s)\bigr)^{-1} \mathcal Z^\vee(t) =  \mathcal Z^\vee(t)  \bigl(\mathcal {\tilde G}(s)\bigr)^{-1}.
\label{eq:G4s}
\end{align}
\end{proposition}
\begin{proof} The equations in \eqref{eq:G1} follow from the equation on the right in  \eqref{eq:3pp10}. To obtain \eqref{eq:G2}, start with the fifth displayed equation in
Lemma \ref{lem:redrel2}. In this equation, multiply each term on the left by $\bigl( \mathcal {\tilde G}(s) \bigr)^{-1}$ and on the right by $\bigl( \mathcal {\tilde G}(s) \bigr)^{-1}$.
In the resulting equation, eliminate $\bigl( \mathcal {\tilde G}(s) \bigr)^{-1} \mathcal W^-(s)$  using \eqref{eq:gi1} and then solve for $\bigl(\mathcal {\tilde G}(s)\bigr)^{-1} \mathcal W^-(t)$ to get \eqref{eq:G2}.
To obtain \eqref{eq:G3}, start with the last displayed equation in
Lemma \ref{lem:redrel2}. In this equation, multiply each term on the left by $\bigl( \mathcal {\tilde G}(s) \bigr)^{-1}$ and on the right by $\bigl( \mathcal {\tilde G}(s) \bigr)^{-1}$.
In the resulting equation, eliminate $\bigl( \mathcal {\tilde G}(s) \bigr)^{-1} \mathcal W^+(s)$  using \eqref{eq:gi2} and then solve for $\bigl(\mathcal {\tilde G}(s)\bigr)^{-1} \mathcal W^+(t)$ to get \eqref{eq:G3}.
 Equation \eqref{eq:G4s} holds since $\mathcal Z^\vee(t)$ is central.
\end{proof}

\noindent  Line \eqref{eq:EsEt} and Propositions \ref{prop:New}, \ref{prop:GFwang}, \ref{prop:wangGF} contain
some relations involving  $E^\pm(t)$, $E(t)$. These relations can be recovered using Theorem  \ref{prop:back1} along with 
Lemmas
 \ref{lem:redrel2}, 
 \ref{lem:gi}, 
 \ref{lem:sw} and
Proposition \ref{thm:mgi}.
The  calculations are routine and omitted.
Lemmas 
\ref{lem:BPhi}, \ref{prop:pbwRelP} 
 and Corollary  \ref{prop:pbwRel} can be obtained using Remark \ref{rem:follow}. They can also be obtained using Theorem  \ref{prop:back1} along with Proposition \ref{thm:mgi} and the following results.

\begin{lemma}
\label{lem:redrel1}
For the algebra $\mathcal U^+_q$,
\begin{align*}
\mathcal W_0 \mathcal G(t) &= q^{-2} \mathcal G(t) \mathcal W_0 + (1-q^{-2})\mathcal W^-(t),
\\
\mathcal W_0 \mathcal W^-(t)&=
\mathcal W^-(t) \mathcal W_0,
\\
\mathcal W^+(t) \mathcal W_0 
&= 
\mathcal W_0 \mathcal W^+(t)+(1-q^{-2})t^{-1} 
\bigl(\mathcal G(t)-\mathcal {\tilde G}(t)\bigr),
\\
\mathcal {\tilde G}(t) \mathcal W_0 
&= 
q^{-2} \mathcal W_0 \mathcal {\tilde G}(t)
+(1-q^{-2})
\mathcal W^-(t)
\end{align*}
\noindent and
\begin{align*}
\mathcal W_1 \mathcal G(t) 
&= 
q^2 \mathcal G(t)  \mathcal W_1 +(1-q^{2})
\mathcal W^+(t),
\\
\mathcal W_1 \mathcal W^-(t) 
&= 
\mathcal W^-(t) \mathcal W_1+(1-q^{-2})t^{-1} 
\bigl(\mathcal G(t)-\mathcal {\tilde G}(t)\bigr),
\\
\mathcal W^+(t) \mathcal W_1&=
\mathcal W_1 \mathcal W^+(t),
\\
 \mathcal {\tilde G}(t) \mathcal W_1
&= 
q^2 \mathcal W_1 \mathcal {\tilde G}(t)  +(1-q^{2})
\mathcal W^+(t),
\end{align*}
\end{lemma}
\begin{proof} Use 
\eqref{eq:3pp1}--\eqref{eq:3pp4}.
\end{proof}

\begin{corollary}
\label{cor:mgi}
For the algebra $\mathcal U^+_q$,
\begin{align}
&\bigl(\mathcal {\tilde G}(t)\bigr)^{-1} \mathcal W_0 = q^2 \mathcal W_0 \bigl(\mathcal {\tilde G}(t)\bigr)^{-1}- q(q-q^{-1})   \mathcal W^-(q^{-2}t) \bigl(\mathcal {\tilde G}(q^{-2}t)\bigr)^{-1} \bigl(\mathcal {\tilde G}(t)\bigr)^{-1}, \label{eq:G5}
\\
 & \bigl(\mathcal {\tilde G}(t)\bigr)^{-1}\mathcal W_1=q^{-2}  \mathcal W_1 \bigl(\mathcal {\tilde G}(t)\bigr)^{-1}  + q^{-1}(q-q^{-1})\mathcal W^+(q^{2}t)  \bigl(\mathcal {\tilde G}(q^{2}t)\bigr)^{-1} \bigl(\mathcal {\tilde G}(t)\bigr)^{-1}.
 \label{eq:G6}
\end{align}
\end{corollary}
\begin{proof} Set $s=t'$ and $t=0$ in \eqref{eq:G2}, \eqref{eq:G3}. Evaluate the results using $\mathcal W^-(0)=\mathcal W_0$ and  $\mathcal W^+(0)=\mathcal W_1$ and $\mathcal {\tilde G}(0)=1$.
\end{proof}

\section{Acknowledgements}
The author thanks Pascal Baseilhac for many conversations about $U^+_q$ and $\mathcal U^+_q$.

\section{Appendix A: An earlier PBW basis for $\mathcal U^+_q$}

In \cite[Theorem~10.2]{altCE} we gave a PBW basis for $\mathcal U^+_q$. In the present section we recall this PBW basis, and give the corresponding reduction rules.

\begin{lemma} \label{lem:pbwP} {\rm (See \cite[Theorem~10.2]{altCE}.)} A PBW basis for $\mathcal U^+_q$ is obtained by its alternating generators in any linear order $<$ such that
\begin{align}
\mathcal W_{-i} < \mathcal G_{j+1} < \mathcal {\tilde G}_{k+1} < \mathcal W_{\ell+1}\qquad \qquad i,j,k, \ell \in \mathbb N. \label{eq:orderP}
\end{align}
\end{lemma}

\noindent For the above PBW basis, the nontrivial reduction rules are a consequence of the following result.
\begin{lemma}
\label{lem:redrel2A} {\rm(See \cite[Lemma~A.6]{altCE}.)}
For the algebra $\mathcal U^+_q$ we have
\begin{align*}
\mathcal W^+(s) \mathcal W^-(t) &= 
\mathcal W^-(t) \mathcal W^+(s) +
(1-q^{-2})\frac{\mathcal G(s) \mathcal {\tilde G}(t)-\mathcal G(t)
\mathcal {\tilde G}(s)}{s-t},
\\
\mathcal {\tilde G}(s) \mathcal G(t) &= 
\mathcal G(t) \mathcal {\tilde G}(s) +
(1-q^{2})st \frac{\mathcal W^-(t) \mathcal W^{+}(s)-\mathcal W^-(s)
\mathcal W^+(t)}{s-t}
\end{align*}
\noindent and also
\begin{align*}
\mathcal G(s) \mathcal W^-(t) &= 
q \frac{(qs-q^{-1}t)\mathcal W^-(t) \mathcal G(s) - (q-q^{-1})s 
\mathcal W^-(s) \mathcal G(t)}{s-t},
\\
\mathcal W^+(s) \mathcal G(t) &= 
q \frac{(q^{-1}s-qt)\mathcal G(t) \mathcal W^+(s) + (q-q^{-1})t 
\mathcal G(s) \mathcal W^+(t)}{s-t},
\\
\mathcal {\tilde G}(s) \mathcal W^-(t) &= 
q^{-1} \frac{(q^{-1}s-qt)\mathcal W^-(t) \mathcal {\tilde G}(s) + (q-q^{-1})s 
\mathcal W^-(s) \mathcal {\tilde G}(t)}{s-t},
\\
\mathcal W^+(s) \mathcal {\tilde G}(t) &= 
q^{-1} \frac{(qs-q^{-1}t)\mathcal {\tilde G}(t) \mathcal W^+(s) - (q-q^{-1})t 
\mathcal {\tilde G}(s) \mathcal W^+(t)}{s-t}.
\end{align*}
\end{lemma}

\noindent Next we give the nontrivial reduction rules for the PBW basis in Lemma \ref{lem:pbwP}.

 \begin{lemma} \label{lem:rr1} For the algebra $\mathcal U^+_q$ the following hold for $i,j \in \mathbb N$:
 \begin{align*}
 \mathcal W_{i+1} \mathcal W_{-j} &=\mathcal W_{-j} \mathcal W_{i+1} 
 + q^{-1}(q-q^{-1})\sum_{\ell=0}^{{\rm min}(i,j)} \bigl( \mathcal G_{i+j+1-\ell} \mathcal {\tilde G}_\ell - \mathcal G_\ell \mathcal {\tilde G}_{i+j+1-\ell} \bigr),
 \\
 \mathcal {\tilde G}_{i+1} \mathcal G_{j+1} &= \mathcal G_{j+1} \mathcal {\tilde G}_{i+1} +
 q(q-q^{-1}) \sum_{\ell=0}^{{\rm min}(i,j)} \bigl( 
 \mathcal W_{\ell-i-j-1} \mathcal W_{\ell+1} - \mathcal W_{-\ell} \mathcal W_{i+j+2-\ell} \bigr)
 \end{align*}
 \noindent and
 \begin{align*}
 \mathcal G_{i+1} \mathcal W_{-j} &= \mathcal W_{-j} \mathcal G_{i+1} +
 q(q-q^{-1}) \sum_{\ell=0}^{{\min}(i,j)} \bigl( 
 \mathcal W_{-\ell} \mathcal G_{i+j+1-\ell} - \mathcal W_{\ell-i-j-1} \mathcal G_\ell\bigr),
 \\
 \mathcal W_{i+1} \mathcal G_{j+1} &= \mathcal G_{j+1} \mathcal W_{i+1} 
 + 
 q(q-q^{-1}) \sum_{\ell=0}^{{\rm min}(i,j)} \bigl( \mathcal G_{i+j+1-\ell} \mathcal W_{\ell+1} - \mathcal G_\ell \mathcal W_{i+j+2-\ell}\bigr),
 \\
  \mathcal {\tilde G}_{i+1} \mathcal W_{-j} &= \mathcal W_{-j} \mathcal {\tilde G}_{i+1} +
 q^{-1}(q-q^{-1}) \sum_{\ell=0}^{{\min}(i,j)} \bigl( 
 \mathcal W_{\ell-i-j-1} \mathcal {\tilde G}_{\ell} - \mathcal W_{-\ell} \mathcal {\tilde G}_{i+j+1-\ell}\bigr),
 \\
 \mathcal W_{i+1} \mathcal {\tilde G}_{j+1} &= \mathcal {\tilde G}_{j+1} \mathcal W_{i+1} 
 + 
 q^{-1}(q-q^{-1}) \sum_{\ell=0}^{{\rm min}(i,j)} \bigl( \mathcal {\tilde G}_{\ell} \mathcal W_{i+j+2-\ell} - \mathcal {\tilde G}_{i+j+1-\ell} \mathcal W_{\ell+1}\bigr).
 \end{align*}
 
 \end{lemma}
 \begin{proof} These relations are obtained by unpacking the equations in Lemma \ref{lem:redrel2A}.
 \end{proof}



%

\bigskip

\noindent Paul Terwilliger \hfil\break
\noindent Department of Mathematics \hfil\break
\noindent University of Wisconsin \hfil\break
\noindent 480 Lincoln Drive \hfil\break
\noindent Madison, WI 53706-1388 USA \hfil\break
\noindent email: {\tt terwilli@math.wisc.edu }\hfil\break

\end{document}